\documentclass[a4paper,11pt]{article}

\usepackage{amsfonts, amsmath, amssymb, amsthm}
\usepackage[colorlinks=true,citecolor=blue]{hyperref}

\usepackage{thm-restate}
\usepackage{cleveref}

\usepackage{algpseudocode}

\usepackage{lineno}

\usepackage{algorithm}
\usepackage{algpseudocode,float}

\usepackage[margin = 2.50cm]{geometry}
\usepackage{mathtools}
\usepackage{graphicx}
\usepackage{enumerate}
\usepackage{verbatim}
\usepackage{tikz}
\usepackage[utf8]{inputenc}
\usepackage{microtype}
\usepackage{amsmath,amsthm,amssymb,color,tikz,mathtools}
\usepackage{wrapfig}

\usepackage{tcolorbox}
\usepackage{mdframed}

\usetikzlibrary{positioning}
\usetikzlibrary{shapes}
\usetikzlibrary{decorations.pathreplacing}
\usepackage{enumerate,enumitem}

\newcommand{\NN}{\mathbb{N}}

\newcommand{\RR}{\mathbb{R}}

\newcommand{\cF}{\mathcal{F}}
\newcommand{\cG}{\mathcal{G}}

\newcommand{\cS}{\mathcal{S}}

\usepackage{bbm}

\renewcommand{\epsilon}{\varepsilon}

\newtheorem{theorem}{Theorem}[section]

\newtheorem{lemma}[theorem]{Lemma}
\newtheorem{claim}[theorem]{Claim}

\title{
A geometric proof of the infinite $(p,q)$-theorem\\
for hyperplane piercing
}

\author{
Sutanoya Chakraborty\footnote{Indian Statistical Institute, Kolkata, India}
\and
Arijit Ghosh\footnotemark[1] 
\and 
Soumi Nandi\footnote{The Institute of Mathematical Science, Chennai, India}
}

\date{}

\begin{document}

\maketitle

\begin{abstract}
We provide a geometric proof of the $(\aleph_{0}, d+1)$-theorem for piercing compact connected sets by hyperplanes. Our argument uses only elementary properties of convex sets and clarifies the core geometric structure behind the theorem.
\end{abstract}

\section{Introduction}
Given two families $\cF$ and $\cG$ of sets, a set $S\in\cG$ is a {\em transversal} of a subset $\cF'\subseteq\cF$
if $S$ intersects every set in $\cF'$.
A subset $\cS\subseteq\cG$ is a {\em transversal of $\cF'$ with respect to $\cG$} if every set in $\cF'$ 
intersects some set in $\cS$.
Given $d\in\NN$ and $k\in\{0,\dots,d-1\}$, if $\cF$ is a family of sets in $\RR^d$ and $\cG$ is the set of all $k$-flats in $\RR^d$, then a subset $\cS\subseteq\cG$ is a $k$-transversal of $\cF'\subset\cF$ if every set in $\cF'$
intersects some set in $\cS$.
In addition, $\cS$ is called a finite $k$-transversal of $\cF'$ if $\cS$ is finite, a {\em point transversal} of $\cF'$
if $k=0$, a {\em line transversal} if $k=1$, and a {\em hyperplane transversal} if $k=d-1$.


The Alon–Kleitman $(p,q)$-theorem~\cite{alon1992piercing,AlonK97} is a deep extension of Helly's theorem. It states that given $p, q, d \in \mathbb{N}$ with $p\geq q \geq d+1$, for a collection $\mathcal{F}$ of compact convex sets in $\mathbb{R}^d$, if every $p$-tuple of sets in $\mathcal{F}$ contains $q$ sets that admit a common piercing point, then there exists an integer $c = c(p,q,d)$ such that $\mathcal{F}$ has a transversal of size $c$. Alon and Kalai~\cite{AlonK95} established the $(p,q)$-theorem for piercing convex sets with hyperplanes. Later, Alon, Kalai, Matoušek, and Meshulam~\cite{alon2002transversal} showed the impossibility of extending the $(p,q)$-theorem to piercing convex sets in $\mathbb{R}^d$ with $k$-flats when $0 < k < d-1$. In the same work, they also proved a more general result for {\em good covers} and {\em Leray complexes}, which implies the Alon–Kleitman $(p,q)$-theorem. For an introduction to this topic, see the surveys by Eckhoff~\cite{Eckhoff2003Survey} and Bárány and Kalai~\cite{BaranyKalai2022HellyTypeProblems}.

Keller and Perles \cite{KellerP22} introduced the notion of $(\aleph_0,q)$-property along the lines of $(p,q)$-property.
A family $\cF$ of sets in $\RR^d$ satisfies the $(\aleph_0,q)$-property with respect to another family $\cG$ of
sets if every infinite subset of $\cF$ contains $q$ sets that have a nonempty intersection with a set in $\cG$.
In a later paper~\cite{KellerP23}, they proved the following result:
\begin{theorem}[Keller and Perles]\label{thm:KP}
    Let $d\in\NN$ and $k\in\{0,\dots,d-1\}$.
    For a $\rho\geq 1$, define a set $B\subset\RR^d$ to be a {\em near-ball with parameter $\rho$} if 
    there is a point $b\in B$ and two positive numbers $r,R$ such that $\overline{B}(b,r)\subseteq B
    \subseteq \overline{B}(b,R)$, $\rho r\geq R$, and $\rho + r\geq R$.
    Then, if $\cF$ is a family of compact near-balls with parameter $\rho$ in $\RR^d$ and $\cF$ satisfies 
    the $(\aleph_0,k+2)$-property with respect to $k$-flats, then $\cF$ has a finite $k$-transversal.
\end{theorem}
Jung and P\'{a}lv\"{o}lgyi~\cite{jung2024noteinfiniteversionspqtheorems} developed a general framework for geometric families, which shows that one can obtain an $(\aleph_0,q)$-theorem whenever the corresponding $(p,q)$-theorem and fractional Helly theorems are available.  
As consequences, they deduced $(\aleph_0,d+1)$-theorems for piercing compact convex sets with points and compact connected sets with hyperplanes. Jung and P\'{a}lv\"{o}lgyi~\cite{jung2024noteinfiniteversionspqtheorems,JungP25} also provided an alternative proof of Theorem~\ref{thm:KP} in the special case of closed balls. Chakraborty, Ghosh, and Nandi~\cite{chakraborty2025finitektransversalsinfinitefamilies} established an $(\aleph_{0}, k+2)$-theorem for piercing compact $\rho$-fat convex sets with $k$-flats.

\section{Our results}
We provide a geometric proof of the following result:

\begin{theorem}\label{thm:hyperplane}
    Let $d\in\mathbb{N}$ and $\cF$ be a family of compact connected sets in $\RR^d$.
    If every infinite subset of $\cF$ contains $d+1$ sets that are pierced by a hyperplane, then $\cF$ has a finite hyperplane transversal.
\end{theorem}
This theorem was proved in \cite{jung2024noteinfiniteversionspqtheorems} using the corresponding
$(p,d+1)$-theorem together with the fractional Helly theorem. Our proof avoids these tools and is closer in spirit to \cite{KellerP23}. Instead, it uses a direct geometric idea based on a simple consequence of the {\em hyperplane separation theorem}: if a hyperplane does not intersect a connected set, then the set lies completely in one of the two half-spaces defined by the hyperplane.
Looking at the problem in this geometric way makes it easier to see the main structural features of the result and shows how the statement follows naturally from basic properties of convex sets. Additionally, using the techniques from \cite{chakraborty2025finitektransversalsinfinitefamilies}, the result can also be extended to its {\em colorful version}, which we state below.
\begin{theorem}\label{thm:hyperplane_colorful}
    Let $d\in\mathbb{N}$ and $\{\cF_n\}_{n\in\mathbb{N}}$ be a sequence of families of compact connected sets such that every sequence $\{B_n\}_{n\in\mathbb{N}}$, where $B_n\in\cF_{m_n}$ for some strictly increasing sequence $\{m_n\}_{n\in\mathbb{N}}$ in $\NN$, contains $d+1$ sets that are pierced by a hyperplane.
    Then there is an $N\in\mathbb{N}$ such that there is a finite collection of hyperplanes that is a transversal for every $\cF_n$ with $n\geq N$.
\end{theorem}

\section{Definitions and Notations}
\begin{itemize}
  \item For any $n\in\mathbb{N}$, $[n]$ denotes the set $\{1,\dots,n\}$.
  \item The origin in $\mathbb{R}^d$ is denoted by $O$.
  \item For any set $A\subseteq\mathbb{R}^d$, $\overline{A}$ denotes the closure and $A^{o}$ denotes the interior of $A$ in the standard topology of $\mathbb{R}^d$,  
  \item Given a point $p\in\mathbb{R}^d$ and $\epsilon>0$, $B(p,\epsilon)$ 
    denotes the open ball of radius $\epsilon$ in $\mathbb{R}^d$ centered at $p$.
    $\overline{B}(p,\epsilon)$ denotes the closure of $B(p,\epsilon)$ in the
    standard topology of $\mathbb{R}^d$.
  \item For any $x\in\mathbb{R}^d$, $||x||$ denotes the $L_2$-norm of $x$.
  \item $\mathbb{S}^{d-1}$ denotes the set $\{x\in\mathbb{R}^d\mid ||x||=1\}$.
  \item For any $x\in\mathbb{R}^d$ and a set $S\subseteq\mathbb{R}^d$, $\mathrm{dist}(x,S)
    : = \inf\{||x-b||\;|\;b\in B\}$. 
    For two sets $A$ and $B$ in $\mathbb{R}^d$, $\mathrm{dist}(A,B): =\inf\{||a-b||\mid 
    a\in A,\; b\in B\}$.
  \item \textit{$\mathrm{aff}(S)$} Given a set $S\subseteq\mathbb{R}^d$, $\mathrm{aff}(S)$ denotes the affine space generated by $S$.
  \item \textit{$k$-flat}: A $k$-flat in $\mathbb{R}^d$ is a $k$-dimensional affine space in $\mathbb{R}^d$.
  \item Given a $p\in\mathbb{R}^d$, a {\em ray} is the set $\{p+tv\mid t\geq 0\}$, where $v\in\mathbb{S}^{d-1}$.
  \item Given any set $U\subseteq \mathbb{R}^d$, $\mathrm{pos}(U):=\{tu\mid t\geq 0,\; u\in U\}$.
  \item \textit{Unbounded cone:} Given a $p\in\mathbb{R}^d$ and a set $U\subseteq\mathbb{R}^d$, an \textit{unbounded cone from $p$ through $p+U$} is the set $p+\mathrm{pos}(U)$.
  \item $R(\vec{r},n)$ and $Q(\vec{r},n)$ : Given a ray $\vec{r} = \{p+tv\mid t\geq 0\}$, $v\in\mathbb{S}^{d-1}$,
    let $h$ be the hyperplane perpendicular to $\vec{r}$ that passes through 
    $p+v$.
    let $U_n=\{x\in h\mid ||x-v||\leq\frac{1}{n}\}$.
    We define $R(\vec{r},n):=(p+\mathrm{pos}(U_n))\cap B(p,\frac{1}{n})$ (see Figure~\ref{fig:R}), 
    and $Q(\vec{r},n):=(p+\mathrm{pos}(U_n))\setminus B(p,n)$ (see Figure~\ref{fig:Q}).
    
  \item 
    \textit{$k$-independent}: Given $d\in\mathbb{N}$ and $k\in\{0,\dots,d-1\}$, a family 
    $\mathcal{S}$ of sets in $\mathbb{R}^d$ is $k$-independent if there is 
    no $k$-flat that passes through $k+2$ distinct sets of $\mathcal{S}$.
  
    \item 
        \textit{$k$-collection}: Given $d\in\mathbb{N}$ and $k\in\{0,\dots,d-1\}$, a family
    $\mathcal{F}$ is $k$-independent if $\mathcal{F}$ does not have a finite 
    $k$-transversal.
\end{itemize}

\begin{figure}
    \centering
    \includegraphics[width=0.60\linewidth]{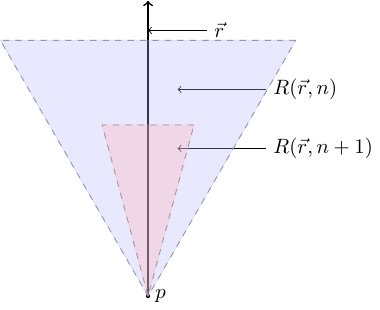}
    \caption{This figure shows $R(\vec{r},n)$ in $\mathbb{R}^2$}
    \label{fig:R}
\end{figure}

\begin{figure}
    \centering
    \includegraphics[width=0.60\linewidth]{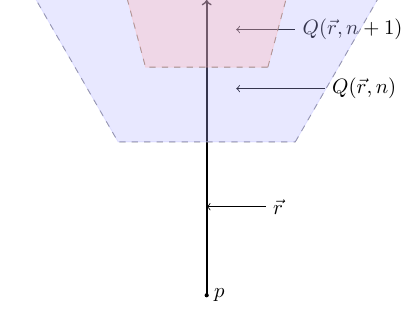}
    \caption{This figure shows $Q(\vec{r},n)$ in $\mathbb{R}^2$}
    \label{fig:Q}
\end{figure}
\section{Proof of Theorem \ref{thm:hyperplane}}
Let $\mathcal{F}$ be a collection of compact connected sets in $\mathbb{R}^d$. 
We show that if $\mathcal{F}$ does not satisfy the $(\aleph_0,d+1)$-property with respect to hyperplanes, then there is an infinite subset
$\mathcal{S}$ of $\mathcal{F}$ such that there is no hyperplane passing through any $d+1$ distinct sets $B_1,\dots,B_{d+1}\in\mathcal{S}$.

Recall that a $(d-1)$-collection is a collection of sets in $\mathbb{R}^d$ which does not have a finite hyperplane transversal, and 
a $(d-1)$-independent collection $\mathcal{S}$ of sets in $\mathbb{R}^d$ is a collection of sets in $\mathbb{R}^d$ such that for distinct $B_1\dots,B_{d+1}\in\mathcal{S}$,
there does not exist a $(d-1)$-flat that pierces every $B_i$, $i\in[d+1]$.

So, in other words, we will show that a $(d-1)$-collection $\mathcal{F}$ of compact connected sets in $\mathbb{R}^d$ has an infinite subset 
$\mathcal{S}$ that is $(d-1)$-independent.

\begin{claim}\label{cl:convergence_point}
 
 Let $\mathcal{F}$ be a $(d-1)$-collection of compact connected sets.
 There is a point $p\in\mathbb{R}^d$ such that one of the following statements is true:
   \begin{enumerate}
     \item\label{ite:p_1} for every bounded open set $U\subset\mathbb{R}^d$ containing $p$, there is a $(d-1)$-collection $\mathcal{F}_U\subset\mathcal{F}$ such that $\forall B\in\mathcal{F}_U$,
       $B\subset U$.
     \item\label{ite:p_2} for every bounded open set $U\subset\mathbb{R}^d$ containing $p$, there is a $(d-1)$-collection $\mathcal{F}_U\subset\mathcal{F}$ such that $\forall B\in\mathcal{F}_U$,
       $B\subset\mathbb{R}^d\setminus U$.
   \end{enumerate}

\end{claim}

\begin{proof}
 
 Let, for every $n\in\mathbb{N}$, $A_n=[-n,n]^d$.
 If there is an $m\in\mathbb{N}$ for which the collection $\mathcal{F}_m := \{B\in\mathcal{F}\mid B\cap A_m^{o}\neq\emptyset\}$ is a $(d-1)$-collection,
 then we shall show that there is a $p\in A_m$ for which \ref{ite:p_1} holds.
 If there is no such $m\in\mathbb{N}$, then clearly \ref{ite:p_2} holds, where $p$ is any point in $\mathbb{R}^d$.

 So let $\mathcal{F}_m$ be a $(d-1)$-collection.
 We split $A_m$ into $2^d$ boxes with axis-parallel hyperplanes $h_1,\dots,h_d$
 such that every box has a diameter that is at most half of the diameter of $A_m$.
 For at least one of these boxes, let us call it $A_{1,m}$, the collection $\mathcal{F}_{1,m} := \{B\in\mathcal{F}_m\mid B\subset A_{1,m}^{o}\}$ 
 is a $(d-1)$-collection.
 That is because, the sets in $\mathcal{F}_m$ that intersect $h_i$ for some $i\in[d]$ or the facets of $A_m$ cannot form a $(d-1)$-collection,
 because the collection of these sets are pierced by finitely many hyperplanes.
 So, the sets of $\mathcal{F}_m$ that lie in the interiors of the $2^d$ boxes created by $h_1,\dots,h_d$ intersecting $A_m$ must form a $(d-1)$-collection, which means that the sets of $\mathcal{F}_m$ that lie in at least one of the $2^d$ boxes form a $(d-1)$-collection.
 We denote such a box by $A_{1,m}$.
 We split $A_{1,m}$ into $2^d$ boxes with at most half the diamter of $A_{1,m}$ as before, and obtain $A_{2,m}$ and $\mathcal{F}_{2,m}$, as before, where $\mathcal{F}_{2,m}$ denotes the collection of sets in $\mathcal{F}_m$ that lie in the interior of $A_{2,m}$.

 Continuing this way, we get a nested sequence $\{A_{n,m}\}_{n\in\mathbb{N}}$ of boxes where $A_{n+1,m}$ has a diameter that is at most half that of $A_{n,m}$ for every $n\in\mathbb{N}$, 
 and corresponding to each $A_{n,m}$, we get a $(d-1)$-collection $\mathcal{F}_{n,m}$ such that $\mathcal{F}_{n+1,m}\subseteq\mathcal{F}_{n,m}$,
 and $\forall B\in\mathcal{F}_{n,m}$, we have $B\subset A_{n,m}^{o}$.
 Therefore, we have $\{p\}=\cap_{n\in\mathbb{N}} A_{n,m}$ for some $p\in\mathbb{R}^d$.
 Then, every open set containing $p$ contains $A_{n,m}$ for some $n\in\mathbb{N}$, which concludes the proof.
\end{proof}

Now we prove Theorem \ref{thm:hyperplane} when $d=1$.

\begin{claim}\label{cl:0th_case}
    Let $\cF$ be a collection of compact connected sets in $\mathbb{R}$.
    In other words, $\cF$ is a collection of closed and bounded intervals in $\RR$.
    Then if $\cF$ does not have a finite point transversal, then $\cF$ contains an infinite $0$-independent 
    sequence.
\end{claim}

\begin{proof}
    By Claim \ref{cl:convergence_point}, there is a point $p\in\mathbb{R}$
    such that either \ref{ite:p_1} or \ref{ite:p_2} holds.
    If \ref{ite:p_1} holds, then for every $n\in\mathbb{N}$, there is an $0$-collection $\mathcal{F}_n\subset\mathcal{F}$ such that $\forall B\in\mathcal{F}_n$,
    we have $B\subset (p-\frac{1}{n},p+\frac{1}{n})$.
    Choose $B_1\in\mathcal{F}_1$ such that $p\notin B_1$. Then there is an $n_2\in\mathbb{N}$ for which $(p-\frac{1}{n_2},p+\frac{1}{n_2})\cap B_1=\emptyset$.
    Now choose $B_2\in\mathcal{F}_{n_2}$ such that $p\notin B_2$.
    In general, given $B_1,\dots,B_r\in\mathcal{F}$ such that $p\notin B_i$ for all $i\in[r]$, choose $n_{r+1}\in\mathbb{N}$ such that
    $(p-\frac{1}{n_{r+1}},p+\frac{1}{n_{r+1}})\cap B_j=\emptyset$, and choose $B_{r+1}\in\mathcal{F}_{n_{r+1}}$ such that $p\notin B_{r+1}$.
    Continuing this way, we obtain the sequence $\{B_n\}_{n\in\mathbb{N}}\subset\mathcal{F}$ such that $B_i\cap B_j=\emptyset$ for all $i,j\in\mathbb{N}$,
    $i\neq j$.
    If $\ref{ite:p_2}$ holds, then for every $n\in\mathbb{N}$, there is an $0$-collection $\mathcal{F}_n\subset\mathcal{F}$ such that $\forall B\in\mathcal{F}_n$,
    we have $B\cap (-n,n)=\emptyset$.
    Therefore, if we have $B_1,\dots,B_r\in\mathcal{F}$ such that $B_i\cap B_j=\emptyset$ whenever $i\neq j$, $i,j\in[r]$, we have an $n_{r+1}\in\mathbb{N}$ 
    for which $B_i\subset (-n_{r+1},n_{r+1})$ $\forall i\in[r]$.
    This means we have an $0$-collection $\mathcal{F}_{n_{r+1}}\subset\mathcal{F}$, and for any $B_{r+1}\in\mathcal{F}_{n_{r+1}}$, we have 
    $B_i\cap B_{r+1}=\emptyset$.
    This shows how we can extend any finite $0$-independent sequence to an infinite one in $\mathcal{F}$.
\end{proof}

\begin{claim}\label{cl:convergence_line}
    Let $d\in\mathbb{N}$ and $\mathcal{F}$ be a $(d-1)$-collection of compact connected sets in $\mathbb{R}^d$.
    Then there is a point $p\in\mathbb{R}^d$ and a ray $\vec{r}$ from $p$ such that at least one of the following holds:
    \begin{enumerate}
        \item\label{ite:R} for every $n\in\mathbb{N}$, the set $\{B\in\mathcal{F}\mid B\subset R(\vec{r},n)\}$ is a $(d-1)$-collection.
        \item\label{ite:Q} for every $n\in\mathbb{N}$, the set $\{B\in\mathcal{F}\mid B\subset Q(\vec{r},n)\}$ is a $(d-1)$-collection.
    \end{enumerate}
\end{claim}

\begin{proof}
    Since $\mathcal{F}$ is a $(d-1)$-collection of compact connected sets, Claim \ref{cl:convergence_point} holds.
    Let $p$ be the point described in Claim \ref{cl:convergence_point}.
    Without loss of generality, let $p$ be the origin $O$.
    For every $n\in\mathbb{N}$, let $\mathcal{F}_n$ be a subset of $\mathcal{F}$ where $\mathcal{F}_n\supseteq\mathcal{F}_{n+1}$, and $\mathcal{F}_n$ is a $(d-1)$-collection for every $n\in\mathbb{N}$.
    For every $n\in\mathbb{N}$, $\mathbb{R}^d$ can be written as a finite union of closed unbounded cones $C^n_1,\dots,C^n_{m_n}$ from the origin that are intersections of finitely many closed half-planes, and each $C^n_i$ is contained in some $\mathrm{pos}(B(v_i,\frac{1}{2^n}))$, $v_i\in\mathbb{S}^{d-1}$, $i\in[m_n]$, . 
    The subset of $\mathcal{F}_n$ that consists of sets that intersect the boundaries of some $C^n_i$, $i\in[m_n]$, cannot be a $(d-1)$-collection, as each $C^n_i$ is the intersection of finitely many half-planes.
    Therefore, at least for one of the cones $C^n_i$, the collection $\{B\in\mathcal{F}_n\mid B\subset {C^n_i}^o\}$ is a $(d-1)$-collection.
    Let $I_n\subseteq[m_n]$ denote the set of all $i\in[m_n]$ for which the collection $\{B\in\mathcal{F}_n\mid B\subset {C^n_i}^o\}$ is a $(d-1)$-collection.
    Let $V_n:=\cup_{i\in I_n}(\mathbb{S}^{d-1}\cap C^n_i)$.
    Note that for every $n\in\mathbb{N}$, $\cap_{j\in[n]}V_j\neq\emptyset$.
    This can be seen from the fact that if $C$ is an unbounded cone in $\mathbb{R}^d$ from the origin, $n\in\mathbb{N}$, and $\{B\in\mathcal{F}_{n'}\mid B\subset C\}$ is a $(d-1)$-collection for some $n'\geq n$, then there are $i_1,\dots,i_n$ with $i_j\in[m_j]$ $\forall j\in[n]$ such that $C'=C\cap(\cap_{j\in[n]}C^j_{i_j})$ is nonempty, and $\{B\in\mathcal{F}_{n'}\mid B\subset C'\}$ is a $(d-1)$-collection.

    As $V_n$ is compact for each $n\in\mathbb{N}$ and $\cap_{i\in[n]}V_i\neq\emptyset$, we have $\cap_{n\in\mathbb{N}}V_n\neq\emptyset$.
    Choose any $v\in \cap_{n\in\mathbb{N}}V_n\neq\emptyset$.
    Any unbounded cone $C$ from the origin with $v$ in its interior also contains $C^n_i$ for some $n\in\mathbb{N}$ and $i\in I_n$, and hence, $\{B\in\mathcal{F}_n\mid B\subset C\}$ is a $(d-1)$-collection for all $n\in\mathbb{N}$, as $\mathcal{F}_n\supseteq\mathcal{F}_{n+1}$ $\forall n\in\mathbb{N}$.

    Define $\vec{r}:=\{tv\mid t\geq 0\}$.
    If Claim \ref{cl:convergence_point}, \ref{ite:p_1} holds, then define $\mathcal{F}_n:=\{B\in\mathcal{F}\mid B\subset B(O,\frac{1}{n})\}$.
    From the above argument, we have that $\{B\in\mathcal{F}_n\mid B\subset R(\vec{r},n)\}=\{B\in\mathcal{F}_n\mid B\subset \mathrm{pos}(B(v,\frac{1}{n}))\}$ is a $(d-1)$-collection.
    If Claim \ref{cl:convergence_point}, \ref{ite:p_2} holds, then define $\mathcal{F}_n:=\{B\in\mathcal{F}\mid B\subset \mathbb{R}^d\setminus B(O,n)\}$.
    Then, the above reasoning shows that $\{B\in\mathcal{F}_n\mid B\subset Q(\vec{r},n)\}=\{B\in\mathcal{F}_n\mid B\subset \mathrm{pos}(B(v,\frac{1}{n}))\}$ is a $(d-1)$-collection.
\end{proof}

In our next two proofs, we use the following lemma by Keller and Perles \cite{KellerP22}.

\begin{lemma}\label{lem:KP}
    Let $d\in\NN$, $k\in\{0,\dots,d-1\}$, $m\in\NN$ and $\cF$ be a collection of compact convex sets.
    If every finite subset of $\cF$ has a $k$-transversal of size $m$, then $\cF$ has a $k$-transversal of 
    size $m$.
\end{lemma}

\begin{theorem}\label{thm:infinite_pq_hyperplane_bdd}
  Let $\mathcal{F}$ be a $(d-1)$-collection for which Claim \ref{cl:convergence_line}, \ref{ite:R}, holds.
  Then there is a sequence $\{B_n\}_{n\in\mathbb{N}}\subset\mathcal{F}$, such that 
  \begin{enumerate}
    \item\label{ite:point} there is no hyperplane piercing every one of $B_{i_1},\dots,B_{i_{d}},\{p\}$ for distinct $i_j\in\mathbb{N}$, $j\in[d]$, 
      and $p$ denotes the point as per the notations of Claim \ref{cl:convergence_line}, \ref{ite:R}.
    \item\label{ite:hyp} there is no hyperplane piercing every one of $B_{i_1},\dots,B_{i_{d+1}}$ for distinct $i_j\in\mathbb{N}$, $j\in[d+1]$,
  \end{enumerate}
\end{theorem}

\begin{proof}
  
  We have already proved the case where $d=1$ in Theorem \ref{cl:0th_case}.
  So, let us assume that the theorem is true for all $1\leq d<D$.
  Let $\mathcal{F}$ be a $(D-1)$-collection in $\mathbb{R}^D$.
  Since for $\mathcal{F}$, Claim \ref{cl:convergence_line}, \ref{ite:R}, holds, 
  $\forall n\in\mathbb{N}$, there is a subcollection $\mathcal{F}_n\subset\mathcal{F}$ which is a $(D-1)$-collection, and 
  $\mathcal{F}_n:=\{B\in\mathcal{F}\mid B\subset R(\vec{r},n), B\cap\vec{r}=\emptyset\}$, where $\vec{r}$ is the ray in the notation of Claim \ref{cl:convergence_line}, \ref{ite:R}.

  Let, for each $n\in\mathbb{N}$, $\mathcal{S}_n\subset\mathcal{F}_n$ be a finite collection such that if $H_n$ is a hyperplane transversal of $\mathcal{S}_n$,
  then $|H_n|\geq n$.
  Lemma \ref{lem:KP} ensures that such a $\mathcal{S}_n$ exists.
  Then, $\mathcal{S}:=\cup_{n\in\mathbb{N}}\mathcal{S}_n$ is a $(D-1)$-collection.

  Without loss of generality, we can assume that $\forall B\in\mathcal{F}$, $B\cap\vec{r}=\emptyset$.
  Let $h$ be the hyperplane through $p$ that is perpendicular to $\vec{r}$, and let $\pi:\mathbb{R}^D\to h$ be the 
  projection map.
  Note that Claim \ref{cl:convergence_point}, \ref{ite:p_1} holds for $\{\pi(B)\mid B\in\mathcal{S}\}$ 
  as well, where $d=D-1$ and $p$ is the point referred to in Claim \ref{cl:convergence_point}, \ref{ite:p_1}, since $\pi(B(p,\frac{1}{n}))$ is a ball in $h$ centered at $p$ with radius $\frac{1}{n}$. 
  This implies Claim \ref{cl:convergence_line}, \ref{ite:R} for $\{\pi(B)\mid B\in\mathcal{S}\}$, $d=D-1$. 
  Then, as per our induction hypothesis, there is an infinite sequence $\{B'_n\}_{n\in\mathbb{N}}\subset\mathcal{S}$ such that
  $\{\pi(B'_n)\}_{n\in\mathbb{N}}$ is $(D-2)$-independent, and there is no $(D-2)$-flat passing through $\{p\},\pi(B'_{i_1}),\dots,\pi(B'_{i_{D-1}})$
  for any distinct $i_1,\dots,i_{D-2}\in\mathbb{N}$.
  Without loss of generality, assume that $\forall n\in\mathbb{N}$, $B'_n\in\mathcal{F}_n$.
  
  Define $B_1:=B'_1,\dots,B_{D-1}:=B'_{D-1}$.
  We show that we can extend $B_1,\dots,B_{D-1}$ to an infinite $(D-1)$-independent sequence.
  Consider the general case: let $B_1,\dots,B_m\in\{B'_n\}_{n\in\mathbb{N}}$ be distinct sets which satisfy \ref{ite:point} and \ref{ite:hyp},
  where $m\geq D-1$.
  When $m=D-1$, \ref{ite:point} and \ref{ite:hyp} are vacuously true.\\

  \noindent\textit{Step-I}: we show that there is a $n'\in\mathbb{N}$ such that for any $B'_n$, $n\geq n'$, the sequence $B_1,\dots,B_m,B'_n$ satisfies \ref{ite:point}.
  
  Let $B_{i_1},\dots,B_{i_{D-1}}$ be distinct sets with $i_j\in[m]$ $\forall j\in[D-1]$.
  Let $h'$ be a hyperplane through $B_1,\dots,B_{D-1},\{p\}$.
  Note that $\vec{r}\cap h'=\{p\}$, because if not, then $\vec{r}$ would be contained in $h'$, which would mean that $h'$ is perpendicular 
  to $h$ and $\pi(h')$ is a $(D-2)$-flat piercing $\{p\},\pi(B_{i_1}),\dots,\pi(B_{i_{D-1}})$, a contradiction.
  $\vec{r}$ can be written as $\{p+tv\mid v\in\mathbb{S}^{d-1}\}$.
  Note that for every $b_j\in B_{i_j}$, $j\in[D-1]$, the hyperplane through $p,b_1,\dots,b_{D-1}$ is uniquely defined, because if the dimension of $\mathrm{aff}(\{p,b_1,\dots,b_{D-1}\})$ was smaller than $D-1$, we would have had a hyperplane through $p,b_1,\dots,b_{D-1}$ that contains $\vec{r}$.
  Therefore, as $B_{i_1},\dots,B_{i_{D-1}},\{p\}$ are compact sets and no hyperplane through them passes through $p+v$, which is a point on $\vec{r}$, there is a closed ball $\overline{B}(v,\frac{1}{n'})$ such that $p+\mathrm{pos}(\overline{B}(v,\frac{1}{n'}))\cap h'=\{p\}$ for every hyperplane $h'$ that passes through $B_{i_1},\dots,B_{i_{D-1}},\{p\}$.
  Since there are only finitely many choices for distinct $i_1,\dots,i_{D-1}$ from $[m]$, we can assume, without loss of generality, that this $n'$ works for every choice of distinct $i_1,\dots,i_{D-1}$ chosen from $[m]$.
  Since for every $n\in\mathbb{N}$, $n\geq n'$, we have $R(\vec{r},n)\subset p+\mathrm{pos}(\overline{B}(v,\frac{1}{n'}))$, for any $B'_n$ with $n\geq n'$, the sequence $B_1,\dots,B'_n$ satisfies \ref{ite:point}.\\
  
  \noindent\textit{Step-II}: we show that there is an $n''\in\mathbb{N}$ such that for any $B'_n$, $n\geq n''$, the sequence $B_1,\dots,B_m,B'_n$ satisfies \ref{ite:hyp}.
  
  If $m=D-1$, then for any choice of $n\in\mathbb{N}$ with $n>D-1$, the statement is vacuously true. In this case, we apply \textit{Step-I} to obtain $B_D$. 
  So, let us assume that $m\geq D$.
  Choose any $B_{i_1},\dots,B_{i_D}$ for distinct $i_j\in[m]$, $j\in[D]$.
  Then, as the sequence $B_1,\dots,B_m$ satisfies \ref{ite:point}, no hyperplane through $B_{i_1},\dots,B_{i_D}$ contains $p$.
  Therefore, there is an $n''\in\mathbb{N}$ such that $B(p,\frac{1}{n})\cap h'=\emptyset$ for every hyperplane $h'$ that passes through $B_{i_1},\dots,B_{i_D}$.
  As there are only finitely many choices for $i_1,\dots,i_D$, without loss of generality, assume that $n'$ works for every choice of $i_1,\dots,i_D$ chosen from $[m]$.
  Then, the sequence $B_1,\dots,B_m,B'_n$ satisfies \ref{ite:hyp} for every $n\in\mathbb{N}$, $n\geq n''$.\\

  Now, let $n_m\in\mathbb{N}$ be such that $n_m>\mathrm{max}(n',n'')$.
  Define $B_{m+1}:=B'_{n_m}$.
  Then $B_1,\dots,B_{m+1}$ satisfies both \ref{ite:point} and \ref{ite:hyp}.
  This shows that we can extend the sequence $B_1,\dots,B_m$.
\end{proof}

In the above proof, we had to show that we can choose sets from $\mathcal{F}$ such that the hyperplanes passing through any distinct $d$ of those
do not contain $p$, because we have $(d-1)$-collections in $\mathcal{F}$ clustering around $p$.
When Claim \ref{cl:convergence_line}, \ref{ite:Q} holds, we have $(d-1)$-collections in $\mathcal{F}$ as far away from $p$ as we want.
So, our proof for this case is simpler.

We state and prove this case now.

\begin{theorem}\label{thm:infinite_pq_hyperplane_unbdd}
  Let $\mathcal{F}$ be a $(d-1)$-collection for which Claim \ref{cl:convergence_line}, \ref{ite:Q}, holds.
  Then there is a $(d-1)$-independent sequence $\{B_n\}_{n\in\mathbb{N}}\subset\mathcal{F}$, such that 
  there is no hyperplane piercing every one of $B_{i_1},\dots,B_{i_{d+1}}$ for distinct $i_j\in\mathbb{N}$, $j\in[d+1]$,
\end{theorem}

\begin{proof}
  We prove the theorem using induction.
  The case where $d=1$ has already been proven as part of Claim \ref{cl:0th_case}.
  So let the theorem hold for all $1\leq d<D$.
  Let $\mathcal{F}$ be a $(D-1)$-collection in $\mathbb{R}^D$ for which Claim \ref{cl:convergence_line}, \ref{ite:Q} holds. 
  Then we have the following:
  $\forall n\in\mathbb{N}$, there is a subcollection $\mathcal{F}_n\subset\mathcal{F}$ which is a $(D-1)$-collection, and 
  $\mathcal{F}_n:=\{B\in\mathcal{F}\mid B\subset Q(\vec{r},n), B\cap\vec{r}=\emptyset\}$, where $\vec{r}$ is the ray in the notation of Claim \ref{cl:convergence_line}, \ref{ite:Q}.

  Let, for each $n\in\mathbb{N}$, $\mathcal{S}_n\subset\mathcal{F}_n$ be a finite collection such that if $H_n$ is a hyperplane transversal of $\mathcal{S}_n$,
  then $|H_n|\geq n$.
  Lemma \ref{lem:KP} ensures that we can find such an $\mathcal{S}_n$.
  Then, $\mathcal{S}:=\cup_{n\in\mathbb{N}}\mathcal{S}_n$ is a $(D-1)$-collection.

  Without loss of generality, we can assume that $\forall B\in\mathcal{F}$, $B\cap\vec{r}=\emptyset$.
  Let $h$ be the hyperplane through $p$ that is perpendicular to $\vec{r}$, and let $\pi:\mathbb{R}^D\to h$ be the 
  projection map.
  Then, $\{\pi(B)\mid B\in\mathcal{S}\}$ is a $(D-2)$-collection.
  Therefore, Claim \ref{cl:convergence_line} applies with $d=D-1$.
  If Claim \ref{cl:convergence_line}, \ref{ite:point} applies, then by Theorem \ref{thm:infinite_pq_hyperplane_bdd}, we get an infinite sequence in $\{\pi(B)\mid B\in\mathcal{S}\}$ that is $(D-2)$-independent.
  If Claim \ref{cl:convergence_line}, \ref{ite:hyp} applies, then too, by applying our induction hypothesis, we get an infinite sequence in $\{\pi(B)\mid B\in\mathcal{S}\}$ that is $(D-2)$-independent.
  So, there is a sequence $\{B'_n\}_{n\in\mathbb{N}}\subseteq\mathcal{S}$ such that
  $\{\pi(B'_n)\}_{n\in\mathbb{N}}$ is $(D-2)$-independent. 
  Without loss of generality, let $B'_n\in\mathcal{F}_n$ $\forall n\in\mathbb{N}$.
  
  Define $B_1:=B'_1,\dots,B_{D}:=B'_{D}$.
  We now show that we can extend the sequence $B_1,\dots,B_D$ to an infinite $(D-1)$-independent sequence.
  Consider the general case where we have the sequence $B_1,\dots,B_m$, $B_i\in\{B'_n\}_{n\in\mathbb{N}}$, of distinct sets, where $m\geq D$, such that there is no hyperplane passing through $B_{i_1},\dots,B_{i_{D+1}}$, for distinct $i_j\in[m]$ $\forall j\in[D+1]$.
  When $m=D$, the claim is vacuously true.
  We show that there is an $n'\in\mathbb{N}$ such that for every $n\in\mathbb{N}$, $n\geq n'$, the sequence $B_1,\dots,B_m,B'_n$ is $(D-1)$-independent. 
  Choose distinct $B_{i_1},\dots,B_{i_D}$, $i_j\in[m]$.
  Then, if $h'$ is a hyperplane passing through $B_{i_1}\dots,B_{i_D}$, then $h'$ cannot be parallel to $\vec{r}$, because if it did, then $\pi(h')$ would be a $(D-2)$-flat in $h$ passing through $D$ distinct sets of $\{\pi(B'_n)\}_{n\in\mathbb{N}}$,
  which is a contradiction.
  In Claim \ref{cl:Q_r}, we shall show that if here is no hyperplane passing through compact sets $B_{i_1},\dots,B_{i_D}$
  that contains $\vec{r}$, then there is an $n_m\in\mathbb{N}$ for which no hyperplane passing through $B_{i_1},\dots,B_{i_D}$
  intersects $Q(\vec{r},n)$ whenever $n\geq n_m$.
  For now, assume that Claim \ref{cl:Q_r} is true.
  Since there are only finitely many choices of distinct $B_{i_1},\dots,B_{i_D}$ chosen from $B_1,\dots,B_m$, assume, without loss of generality, that $n_m$ works for every choice of $B_{i_1},\dots,B_{i_D}$ chosen from $B_1,\dots,B_m$.
  Choose $B_{m+1}\in\mathcal{F}_{n_m}$.
  Then $\{B_1\dots,B_{m+1}\}$ is $(D-1)$-independent, as $B_{m+1}\subset Q(\vec{r},n_m)$.
  This shows that we can extend the sequence $B_1,\dots,B_m$ while maintaining the property of $(D-1)$-independence, which concludes the proof.
\end{proof}

Now all that is left to complete the proof of Theorem \ref{thm:hyperplane} is to prove Claim \ref{cl:Q_r}.

\begin{claim}\label{cl:Q_r}
  Let $d\in\mathbb{N}$ and $B_1,\dots,B_d$ be compact connected sets in $\mathbb{R}^d$ and $\vec{r}$ be a ray from a point $p\in\mathbb{R}^d$
  such that no hyperplane that passes through $B_1,\dots,B_d$ contains a straight line parallel to $\vec{r}$.
  Then there is an $N\in\mathbb{N}$ for which no hyperplane passing through $B_1,\dots,B_d$ intersects $Q(\vec{r},N)$.
\end{claim}

\begin{proof}
  Since for any choice of $b_i\in B_i$, $i\in[d]$, a hyperplane passing through $b_1,\dots,b_d$ does not contain a line parallel to 
  $\vec{r}$, we can conclude that there is a unique hyperplane that passes through $b_1,\dots,b_d$.
  Define $B:=B_1\times\dots\times B_d$, and let, for every $x\in B$, $h_x$ be the hyperplane passing through $x_1,\dots,x_d$, $x=(x_1,\dots,x_d)$.
  Let $l$ be the straight line on which $\vec{r}$ lies.
  Here, note that since $h_x$ is a hyperplane that is not perpendicular to $\vec{r}$, $h_x$ intersects $l$ for every $x\in B$, but does not contain $l$.
  Let $g:B\to\mathbb{R}^{+}\cup\{0\}$ be defined as $g(x)= \text{dist}(h_x\cap l,p)$.
  As $B$ is compact and $g$ is continuous, $g(B)$ is bounded.
  So, there is an $N_1\in\mathbb{N}$ for which $N_1> t$ for every $t\in g(B)$.
  Let $q\in\vec{r}$ such that $\text{dist}(p,q)>N_1$.
  Define $f:B\to (0,\pi)$, where $f(x)$ is the angle $h_x$ makes with any straight line parallel to $\vec{r}$.
  As $B$ is compact, $f(B)$ is closed, so there is a $\theta > 0$ such that for every $\phi\in f(B)$, 
  $\theta < \phi$.
  We can find an $N\in\mathbb{N}$ with $N>N_1$ for which every straight line joining any point in $Q(\vec{r},N)$ 
  with $q$ makes an angle smaller than $\theta$ with $\vec{r}$.
  Therefore, no hyperplane passing through $B_1,\dots,B_d$ intersects $Q(\vec{r},N)$.
\end{proof}

\bibliographystyle{alpha}
\bibliography{references}
\end{document}